\documentclass[11pt,reqno]{amsart}
\usepackage[a4paper,margin=30mm]{geometry}

\usepackage{amsmath, amssymb, amsthm, amsrefs}
\usepackage{mathtools}
\usepackage{mathrsfs}
\usepackage{bbm}
\usepackage{cases}
\usepackage{graphicx} 
\usepackage{svg}
\usepackage{booktabs}
\usepackage{siunitx}
\usepackage{float} 
\usepackage{subfigure} 
\usepackage{url}
\usepackage{enumitem}
\usepackage{color}
\usepackage{xcolor}
\usepackage{epstopdf}
\usepackage{hyperref} 
\newtheorem{theorem}{Theorem}[section]
\newtheorem{lemma}[theorem]{Lemma}
\newtheorem{corollary}[theorem]{Corollary}
\newtheorem{proposition}[theorem]{Proposition}

\theoremstyle{definition}
\newtheorem{conjecture}[theorem]{Conjecture}

\theoremstyle{remark}
\newtheorem{remark}[theorem]{Remark}

\numberwithin{equation}{section}

\newcommand{\dd}{\,\mathrm{d}}
\newcommand{\dist}{\operatorname{dist}}

\providecommand{\eat}[1]{}

\begin{document}

\title[The Weinstock inequality in hyperbolic space]{The Weinstock inequality for convex domains\\ in hyperbolic space}

\author{Jin Sun}
\address{Jin Sun: School of Mathematical Sciences, Fudan University, 200433, Shanghai, China}
\email{\href{mailto:jsun22@m.fudan.edu.cn}{jsun22@m.fudan.edu.cn}}

\author{Lili Wang}
\address{Lili Wang: School of Mathematics and Statistics, Key Laboratory of Analytical Mathematics and Applications (Ministry of Education), Fujian Key Laboratory of Analytical Mathematics and Applications (FJKLAMA), Fujian Normal University, 350117, Fuzhou, China}
\email{\href{mailto:liliwang@fjnu.edu.cn}{liliwang@fjnu.edu.cn}}

\author{Tao Wang}
\address{Tao Wang: Beijing International Center for Mathematical Research, Peking University, 100871, Beijing, China}
\email{\href{mailto:taowang25@pku.edu.cn}{taowang25@pku.edu.cn}}


\begin{abstract}
We prove the Weinstock inequality for the first Steklov eigenvalue of convex domains in hyperbolic space $\mathbb{H}^{n}$, resolving Open Question~4.27 of \cite{CGGS2024} for the remaining case $n=3$. Our argument replaces the global monotonicity required in earlier work \cite{GuLiWan2025} with a one-crossing property, which is established via an explicit slope comparison. The proof works uniformly for all $n\geq 3$.
\end{abstract}

\maketitle

\section{Introduction}

Let $\Omega$ be a compact Riemannian manifold with Lipschitz boundary $\Sigma=\partial\Omega$. Denote by $\nu$ the outward unit normal on $\Sigma$. The Steklov eigenvalue problem, introduced by Stekloff \cite{Stekloff1902}, asks for $\sigma\in \mathbb{R}$ such that
\begin{equation*} 
  \Delta u=0\ \text{ in }\Omega,
  \qquad
  \frac{\partial u}{\partial\nu}=\sigma\,u\ \text{ on }\Sigma,
\end{equation*}
admits a nonzero solution. The Steklov spectrum is discrete. It consists of a sequence of eigenvalues 
\[
 0=\sigma_{0}<\sigma_1\leq\sigma_{2}\leq\cdots\nearrow\infty,
\]
with $\sigma_1$ given variationally by
\begin{equation}\label{eq:rayleigh}
  \sigma_1(\Omega)
  =\min\left\{\frac{\int_{\Omega}|\nabla u|^{2}\dd v}{\int_{\Sigma}u^{2}\dd\mu}
  \ :\
  u\in H^{1}(\Omega)\setminus\{0\},\ \int_{\Sigma}u\dd\mu=0\right\}.
\end{equation}
The problem governs the Dirichlet‑to‑Neumann operator. Hence, $\sigma_1$ is more sensitive to the geometry of $\Sigma$ than to that of $\Omega$, and the
natural shape optimization problem prescribes the boundary area. See \cites{GirouardPolterovich2017,CGGS2024} for surveys.

In Euclidean space, Weinstock \cite{Weinstock1954} proved that among simply connected planar domains with fixed perimeter, the disk uniquely maximizes $\sigma_1(\Omega)$. It is also possible to maximize $\sigma_1(\Omega)$ by fixing the volume, see \cites{Brock2001,Escobar1999}. The Weinstock inequality for convex domains in all dimensions was established by Bucur--Ferone--Nitsch--Trombetti \cite{BFNT2021}, and their result was later extended by Kwong and Wei \cite{KwongWei2023} to star-shaped mean convex domains. In hyperbolic space $\mathbb{H}^{n}$, the fixed-perimeter version appears as Open Question~4.27 of \cite{CGGS2024}. The precise formulation is as follows.
\begin{conjecture}[{\cite{CGGS2024}*{Open Question~4.27}}]\label{conj:main}
Let $\Omega \subset \mathbb{H}^{n}$ be a bounded convex domain and let $\Omega^*$ be a geodesic ball with $|\partial\Omega^*|=|\partial\Omega|$. Is it true that 
\begin{equation}\label{eq:Weinstock_inequality}
    \sigma_1(\Omega)\leq\sigma_1(\Omega^*)
\end{equation}
with equality if and only if $\Omega$ is a geodesic ball?
\end{conjecture}

The case $n=2$ is classical: a convex domain in $\mathbb{H}^{2}$ is simply connected, hence the inequality and its rigidity follow from \cite{FraserSchoen2011}. In higher dimensions the decisive progress is due to Gu-Li-Wan \cite{GuLiWan2025}, who proved Conjecture~\ref{conj:main} for all $n\geq4$. Indeed, they obtained the inequality \eqref{eq:Weinstock_inequality} for a larger class of star-shaped mean convex domains by combining the inverse mean curvature flow with Weinberger's mass transplantation. Their method left $n=3$ open, because they could treat $\mathbb{H}^{3}$ only under an a priori bound on the circumradius \cite{GuLiWan2025}*{Proposition~4.9}, and conjectured in \cite{GuLiWan2025}*{Remark~4.8} that the restriction is merely a limitation of their method.

We stress that Conjecture~\ref{conj:main} is strictly stronger than the fixed-volume statement of \cite{BinoySanthanam2014}: since $\sigma_1$ of a geodesic ball decreases with its radius, and by the isoperimetric inequality the ball with a given perimeter has volume at least as large as that with the same volume, the fixed-perimeter inequality implies the fixed-volume one, but the converse does not hold.

In this paper, we establish the following theorem.
\begin{theorem}\label{thm:main}
Let $n\geq3$ and let $\Omega \subset \mathbb{H}^{n}$ be a bounded convex domain with smooth boundary. Let $\Omega^*=B(R)$ be the geodesic ball with
$|\partial\Omega^*|=|\partial\Omega|$. Then
\begin{equation}\label{eq:main}
\sigma_1(\Omega)\leq\sigma_1(\Omega^*),
\end{equation}
and equality holds if and only if $\Omega$ is a geodesic ball.
\end{theorem}

Convexity is used only to ensure that the boundary is star-shaped and mean convex. The proof actually yields a sharper, scale‑invariant theorem.
\begin{theorem}\label{thm:starshaped}
Let $n\geq3$ and let $\Omega \subset \mathbb{H}^{n}$ be a smooth bounded domain with star-shaped mean convex boundary $\partial \Omega$. If $B(R)$ is the geodesic ball with $|B(R)|=|\Omega|$, then
\begin{equation}\label{eq:main-product}
  |\partial\Omega|\,\sigma_1(\Omega)\leq|\partial B(R)|\,\sigma_1(B(R)),
\end{equation}
and equality holds if and only if $\Omega$ is a geodesic ball.
\end{theorem}

Inequality~\eqref{eq:main} follows from~\eqref{eq:main-product} by the isoperimetric inequality and the monotonicity of $t\mapsto|S(t)|\,\sigma_1(B(t))$. Therefore, Theorem~\ref{thm:starshaped} establishes Theorem~\ref{thm:main}.

\begin{remark}\label{rem:nonsmooth}
The smoothness hypothesis in Theorem~\ref{thm:main} and Theorem~\ref{thm:starshaped} is used only to apply the inverse mean curvature flow in Theorem~\ref{thm:moment}, which requires a $C^{2}$ boundary. Both theorems extend to arbitrary bounded convex domains with Lipschitz boundary by approximation with smooth convex domains, the rigidity statement being unchanged. Therefore, Conjecture~\ref{conj:main} holds in all dimensions.
\end{remark}

In particular, we close the sole remaining open dimension.
\begin{corollary}\label{cor:H3}
Conjecture~\ref{conj:main} holds in $\mathbb{H}^{3}$, without any restriction on the circumradius.
\end{corollary}

As a further consequence of the proof, we obtain an inequality for the harmonic mean of the first $n-1$ Steklov eigenvalues for star-shaped mean convex domains. See Proposition~\ref{prop:harmonic} in Section~\ref{sec:harmonic}.

\begin{remark}\label{rem:idea}
The case $n=3$ was previously inaccessible because the mass‑transplantation argument of \cite{GuLiWan2025} requires a certain radial density to be globally decreasing, a property that fails in $\mathbb{H}^{3}$. We circumvent this by observing that only a one‑crossing property at the comparison radius is needed. The crossing property is established in Lemmas~\ref{lem:slope}
and~\ref{lem:crossing} without any dimensional restriction.
\end{remark}

This paper is organized as follows. Section~\ref{sec:prelim} fixes notation and collects the facts we quote. Section~\ref{sec:radial} contains the two new radial lemmas. In Section~\ref{sec:proof}, we prove Theorems~\ref{thm:main} and~\ref{thm:starshaped}. Section~\ref{sec:harmonic} presents a harmonic mean inequality.

\section{Preliminaries}\label{sec:prelim}

Fix a point $O\in\mathbb{H}^{n}$ and write $r=\dist(O,\cdot)$. In geodesic polar coordinates about $O$,
\begin{equation*} 
  g_{\mathbb{H}^{n}}=\dd r^{2}+\lambda(r)^{2}g_{\mathbb{S}^{n-1}},
  \qquad
  \lambda(r)=\sinh r,
\end{equation*}
so that $\mathbb{H}^{n}=[0,\infty)\times_{\lambda}\mathbb{S}^{n-1}$. Set $\lambda'=\cosh r$ and note the identities $\lambda''=\lambda$ and $(\lambda')^{2}-\lambda^{2}=1$,
used constantly below. Denote the geodesic ball and sphere of radius $r$ centered at $O$ by $B(r)$ and $S(r)=\partial B(r)$, and let $\omega_{n-1}=|\mathbb{S}^{n-1}|$. Then
\begin{equation*} 
  |B(r)|=\omega_{n-1}\int_{0}^{r}\lambda(t)^{n-1}\dd t,
  \qquad
  |S(r)|=\omega_{n-1}\lambda(r)^{n-1}.
\end{equation*}
Throughout, $\Omega$ is a bounded domain with boundary $\Sigma=\partial\Omega$, $\dd v$ the volume element and $\dd\mu$ the induced boundary measure.

The volume-to-perimeter ratio of geodesic balls plays a central role:
\begin{equation*} 
\varphi(r):=\frac{|B(r)|}{|S(r)|}
=\frac{1}{\lambda(r)^{n-1}}\int_{0}^{r}\lambda(t)^{n-1}\dd t.
\end{equation*}
Differentiating directly gives
\begin{align}
\varphi'&=1-(n-1)\frac{\lambda'}{\lambda}\varphi, \label{eq:phi1}\\
\varphi''&=(n-1)\frac{\varphi}{\lambda^{2}}
             -(n-1)\frac{\lambda'}{\lambda}\varphi'=\frac{1-(1+(n-1)(\lambda')^{2})\varphi'}{\lambda\lambda'}. \label{eq:phi2}
\end{align}

By separation of variables, the first Steklov eigenvalue and eigenfunctions of geodesic balls can be expressed in terms of $\varphi$; see \cites{BinoySanthanam2014,GuLiWan2025} for details.
\begin{proposition}\label{prop:ball}
For every $r>0$,
\begin{equation}\label{eq:ballsigma}
  \sigma_1(B(r))=\frac{\varphi'(r)}{\varphi(r)}=(\log\varphi)'(r),
\end{equation}
with multiplicity $n$ and eigenfunctions $\varphi(r)\theta_{i}$, where $\theta_1,\dots,\theta_{n}$ are the restrictions to $\mathbb{S}^{n-1}$ of the Euclidean coordinates. Moreover,
$r\mapsto\sigma_1(B(r))$ is strictly decreasing, while
\begin{equation*} 
  r\longmapsto|S(r)|\,\sigma_1(B(r))
\end{equation*}
is nondecreasing. It is constant when $n=2$ and strictly increasing when $n\geq3$.
\end{proposition}

For a bounded domain $\Omega$ with smooth boundary, the center-of-mass theorem \cites{AithalSanthanam1996,BinoySanthanam2014} supplies an origin $O\in\Omega$ such that the functions $\varphi(r)\theta_{i}$ have zero mean on $\partial\Omega$. These functions satisfy the mean-zero condition in \eqref{eq:rayleigh} and are therefore admissible as test functions. Summing over $i=1,\dots,n$ yields the following estimate.

\begin{proposition}\label{prop:rayleigh}
Let $\Omega\subset\mathbb{H}^{n}$ be a bounded domain with smooth boundary. Then
\begin{equation*} 
\sigma_1(\Omega)\int_{\partial\Omega}\varphi^{2}\dd\mu
\leq\int_{\Omega}\left ((\varphi')^{2}+(n-1)\frac{\varphi^{2}}{\lambda^{2}}\right ) \dd v.
\end{equation*}
\end{proposition}

We collect the basic properties of $\varphi$ needed later, see \cite{GuLiWan2025}*{Section~3} for details.

\begin{lemma}\label{lem:phimono}
The function $\varphi$ is strictly increasing and strictly concave on $(0,\infty)$, with
\[
  \lim_{r\to0}\frac{\varphi}{\lambda}=\frac{1}{n},\qquad
  \lim_{r\to\infty}\varphi=\frac{1}{n-1},\qquad
  \lim_{r\to0}\varphi'=\frac{1}{n},\qquad
  \lim_{r\to\infty}\varphi'=0.
\]
Moreover, $\lambda^{2}\varphi'/\varphi$ is strictly increasing, and for every $r>0$,
\begin{equation*} 
0<\varphi'(r)<\frac{1}{n},\qquad \varphi''(r)<0.
\end{equation*}
\end{lemma}

Define the energy density
\begin{equation*} 
  \mathcal{E}(r):=\varphi'(r)^{2}+(n-1)\frac{\varphi(r)^{2}}{\lambda(r)^{2}},
\end{equation*}
where $\varphi'(r)^{2}$ and $(n-1)\varphi(r)^2/\lambda(r)^2$ are, respectively, the radial and spherical parts of $\mathcal{E}(r)$. From \eqref{eq:phi2} one obtains
\begin{equation*} 
  \mathrm{div}\left (\varphi\varphi'\partial_{r}\right)
  =(\varphi')^{2}+(n-1)\frac{\varphi^{2}}{\lambda^{2}}=\mathcal{E},
\end{equation*}
and integrating over $B(R)$ gives
\begin{equation}\label{eq:E0}
  \int_{B(R)}\mathcal{E}\dd v
  =\varphi(R)\varphi'(R)\,|S(R)|
  =\varphi'(R)\,|B(R)|.
\end{equation}

We need two classical results. The first is the isoperimetric inequality in hyperbolic space, and the second is Weinberger's mass transplantation principle.

\begin{theorem}[{\cite{Schmidt1943}}]\label{thm:isop}
If $\Omega\subset\mathbb{H}^{n}$ is a bounded domain and $|\Omega|=|B(R)|$, then
\[
|\partial\Omega|\geq|S(R)|,
\]
with equality only for geodesic balls.
\end{theorem}

\begin{theorem}[{\cite{Weinberger1956}}; see also {\cite{FreitasLaugesen2021}}]
\label{thm:transplant}
Let $\Omega\subset\mathbb{H}^{n}$ be bounded with $|\Omega|=|B(R)|$ and let $f$ be radial and nonincreasing. Then
\[
\int_{\Omega}f\dd v\leq\int_{B(R)}f\dd v,
\]
and the inequality reverses for nondecreasing $f$. Equality for strictly monotone $f$ forces $|\Omega\,\triangle\,B(R)|=0$.
\end{theorem}

The next ingredient is a weighted volume in the inverse mean curvature flow. Set
\begin{equation}\label{eq:wdef}
  w(r):=\frac{\lambda'(r)}{\lambda(r)}\varphi(r)
       =\frac{1-\varphi'(r)}{n-1},
  \qquad
  \mathcal{W}(\Omega):=\int_{\Omega}w\dd v,
\end{equation}
where the second equality follows from \eqref{eq:phi1}. Note that $w = \frac{\varphi\,H_{S(r)}}{n-1}$ with $H_{S(r)} = (n-1)\frac{\lambda'}{\lambda}$ the mean curvature of the centred geodesic sphere. By Lemma~\ref{lem:phimono}, $w$ increases strictly from $\tfrac{1}{n}$ to $\tfrac{1}{n-1}$. Hence $\mathcal{W}(B(t))$ increases from $0$ to $\infty$ and we may define $\eta:(0,\infty)\to(0,\infty)$ by
\begin{equation}\label{eq:etadef}
  \eta\left (\mathcal{W}(B(t))\right )=\frac{1}{|S(t)|},
  \qquad t>0.
\end{equation}

\begin{lemma}\label{lem:eta}
The function $\eta$ is smooth, decreasing and strictly log-convex, with
\begin{equation}\label{eq:etalog}
  (\log\eta)'\left (\mathcal{W}(B(t))\right )=-\frac{n-1}{|B(t)|}.
\end{equation}
Consequently $\eta^{2}$ is strictly convex, and at $\mathcal{W}_{0}:= \mathcal{W}(B(R))$,
\begin{equation}\label{eq:etatangentdata}
  \eta(\mathcal{W}_{0})=\frac{1}{|S(R)|},
  \qquad
  (\eta^{2})'(\mathcal{W}_{0})
  =-\frac{2(n-1)}{|B(R)|\,|S(R)|^{2}}.
\end{equation}
\end{lemma}

\begin{proof}
Differentiating \eqref{eq:etadef} in $t$ and using
\[
\frac{\dd}{\dd t}\mathcal{W}(B(t))=w(t)|S(t)|,\qquad
\frac{\dd}{\dd t}\log|S(t)|=(n-1)\frac{\lambda'}{\lambda},\qquad
w=\frac{\lambda'\varphi}{\lambda},
\]
we obtain
\[
 (\log\eta)'\left (\mathcal{W}(B(t))\right )
  =\frac{-(n-1)\frac{\lambda'}{\lambda}}{w(t)|S(t)|}
  =-\frac{n-1}{\varphi(t)|S(t)|}
  =-\frac{n-1}{|B(t)|},
\]
which is \eqref{eq:etalog}. The right-hand side increases strictly in $t$, and $\mathcal{W}(B(t))$ increases strictly in $t$, so $\log\eta$ is strictly convex. Writing $\eta^{2}=e^{2\log\eta}$ gives
\[
(\eta^{2})''=2\eta^{2}\left [(\log\eta)''+2((\log\eta)')^{2}\right ]>0.
\]
Finally, $(\eta^{2})'=2\eta^{2}(\log\eta)'$ together with \eqref{eq:etadef} and \eqref{eq:etalog} yields \eqref{eq:etatangentdata}.
\end{proof}

The following estimate, established by Gu--Li--Wan \cite{GuLiWan2025}, is the only place where the curvature flow enters.

\begin{theorem}[\cite{GuLiWan2025}*{Corollary~4.3}]\label{thm:moment}
Let $n\geq3$ and let $\Omega\subset\mathbb{H}^{n}$ be a smooth bounded domain with star-shaped mean convex boundary $\Sigma$. Then
\begin{equation}\label{eq:moment}
  \int_{\Sigma}\varphi^{2}\dd\mu
  \ \geq\ |\Sigma|\,|\Omega|^{2}\,\eta\left (\mathcal{W}(\Omega)\right)^{2}.
\end{equation}
\end{theorem}

Equality in \eqref{eq:moment} holds when $\Omega=B(R)$, by \eqref{eq:etadef}. Under the inverse mean curvature flow, $\Sigma$ evolves into a family of hypersurfaces $\Sigma_{t}$, and we denote by $\Omega_{t}$ the enclosed domain. Gu--Li--Wan proved that the quantity
\[
t\mapsto|\Sigma_{t}|^{-1}\frac{\int_{\Sigma_t}\varphi \dd\mu_t}{|\Omega_{t}|\eta(\mathcal{W}(\Omega_t))}
\]
is nonincreasing with limit at least $1$ \cite{GuLiWan2025}*{Proposition~4.2}. The inequality \eqref{eq:moment} then follows from the Cauchy--Schwarz inequality. We refer to \cite{GuLiWan2025}*{Corollary~4.3} for the proof.

\begin{remark}\label{rem:dimrange}
Although the main theorem of \cite{GuLiWan2025} is stated for $n\geq4$, the proof of \eqref{eq:moment} is valid for every $n\geq3$. It uses only the long-time existence, star-shapedness and asymptotics of the inverse mean curvature flow in $\mathbb{H}^{n}$ due to Gerhardt \cite{Gerhardt2011}, requiring the flowing hypersurface to have dimension $n-1\geq2$. The hypothesis $n\geq4$ is only used subsequently in \cite{GuLiWan2025} for the radial estimates, which we bypass using the methods in Section~\ref{sec:radial} below.
\end{remark}

\section{Crossing estimates}\label{sec:radial}
We now carry out the crossing argument outlined in Remark~\ref{rem:idea}. We rewrite the energy density as
\[
  \mathcal{E}=(\varphi')^{2}+q^{2},
\]
where
\begin{equation}\label{eq:pq}
  q:=\sqrt{n-1}\,\frac{\varphi}{\lambda}.
\end{equation}
Recall that $\varphi'(r)^{2}$ and $q(r)^2$ are the radial and spherical parts of $\mathcal{E}(r)$ respectively. Differentiating and using
\eqref{eq:phi1}--\eqref{eq:phi2} gives
\begin{equation}\label{eq:pqderiv}
  \varphi''=\frac{1-A\varphi'}{\lambda\lambda'},
  \qquad
  q'=\frac{n\varphi'-1}{\sqrt{n-1}\,\lambda},
\end{equation}
where 
\begin{equation*} 
  A=1+(n-1)(\lambda')^{2}=n+(n-1)\lambda^{2}.
\end{equation*}
Lemma~\ref{lem:phimono} implies that both $\varphi'$ and $q$ are positive and strictly decreasing on $(0,\infty)$, with $\varphi'\to0$ and $q\to0$ as $r\to\infty$.

The crossing condition in Lemma~\ref{lem:crossing} below will require a comparison between the slopes of $\varphi'$ and $q$. We begin with a sharp pointwise bound for $\varphi'$.

\begin{lemma}\label{lem:slope}
For all $n\geq 3$ and $r>0$,
\begin{equation}\label{eq:pbound}
  \varphi'(r)\leq p(r)
  :=\frac{n+2\lambda'}{n\left(A+2\lambda'\right)}.
\end{equation}
Consequently
\begin{equation}\label{eq:slope}
  0<\frac{\varphi''(r)}{q'(r)}\leq\frac{2\sqrt{n-1}}{n}<1,
\end{equation}
and in particular $|\varphi''|<|q'|$ pointwise.
\end{lemma}

\begin{proof}
Lemma~\ref{lem:phimono} gives $0<\varphi'<\tfrac{1}{n}$ and $\varphi''<0$. From \eqref{eq:phi2}, $\varphi''<0$ is equivalent to $1-A\varphi'<0$. Hence, for $r>0$, we have $1-n\varphi'>0$ and $A\varphi'-1>0$. Together with \eqref{eq:pqderiv}, this yields
\[
  0<\frac{\varphi''}{q'}
  =\frac{\sqrt{n-1}(A\varphi'-1)}{\lambda'(1-n\varphi')}.
  \]
The inequality $\frac{\varphi''}{q'}\leq\frac{2\sqrt{n-1}}{n}$ is equivalent to $n(A\varphi'-1)\leq2\lambda'(1-n\varphi')$, i.e.\ to
\[\varphi'\left (nA+2n\lambda'\right )\leq n+2\lambda',\]
which is precisely the estimate \eqref{eq:pbound}. The bound $2\sqrt{n-1}/n<1$ is just $(n-2)^2>0$. Thus \eqref{eq:slope} follows from \eqref{eq:pbound}.

To prove \eqref{eq:pbound}, set $p$ as in the statement. A direct computation using $\lambda''=\lambda$ and $(\lambda')^{2}-\lambda^{2}=1$ gives
\[
  p\,'
  =-\frac{2(n-1)\lambda\left((\lambda')^{2}+n\lambda'+1\right)}
         {n\,(A+2\lambda')^{2}},
  \qquad
  \frac{1-Ap}{\lambda\lambda'}
  =-\frac{2(n-1)\lambda}{n\,(A+2\lambda')}.
\]
Hence
\begin{equation}\label{eq:defect}
  p\,'-\frac{1-Ap}{\lambda\lambda'}
  =\frac{2(n-1)(n-2)\,\lambda\lambda'\,(\lambda'-1)}{n\,(A+2\lambda')^{2}}>0
  \qquad\text{for }r>0.
\end{equation}
Now consider $p-\varphi'$. Using the expression for $\varphi''$ from
\eqref{eq:pqderiv}, we have
\[
  (p-\varphi')'
  =p\,'-\frac{1-A\varphi'}{\lambda\lambda'}
  =\left(p\,'-\frac{1-Ap}{\lambda\lambda'}\right)
   -\frac{A}{\lambda\lambda'}(p-\varphi').
\]
A direct calculation yields
\[
  \left(\frac{\lambda^{n}}{\lambda'}\right)' = \frac{n\lambda^{n-1}\lambda' - \lambda^{n+1}}{(\lambda')^2}=\frac{\lambda^n\big(n(\lambda')^2 - \lambda^2\big)}{\lambda(\lambda')^2} = \frac{A\lambda^n}{\lambda(\lambda')^2}.
\]
Multiplying by the integrating factor $\lambda^{n}/\lambda'$ and using \eqref{eq:defect}, we obtain
\[
  \left(\frac{\lambda^{n}}{\lambda'}(p-\varphi')\right)'
  =\frac{\lambda^{n}}{\lambda'}
    \left(p\,'-\frac{1-Ap}{\lambda\lambda'}\right)>0.
\]
As $r\downarrow0$, $\frac{\lambda^n}{\lambda'}\sim r^n\to0$ and $p(0)-\varphi'(0)=0$, so integrating from $0$ to $r$ gives $p(r)-\varphi'(r)>0$ for $r>0$, which is \eqref{eq:pbound}.
\end{proof}

With the slope comparison in hand, we turn to the one-crossing property. Fix $R>0$ and define the \emph{comparison density at level $R$},
\begin{equation*}
\Psi_{R}(r):=\mathcal{E}(r)+2(n-1)\,\varphi'(R)\,w(r),\qquad r>0.
\end{equation*}
The coefficient $2(n-1)$ arises naturally from the tangent line to $\eta^{2}$ in \eqref{eq:etatangentdata}, as shown in the proof of Lemma~\ref{lem:eta}.

\begin{lemma}\label{lem:crossing}
Let $n\geq3$ and $R>0$. Then
\begin{equation}\label{eq:Psiidentity}
  \Psi_{R}(r)-\Psi_{R}(R)
  =\left (\varphi'(r)-\varphi'(R)\right )^{2}+q(r)^{2}-q(R)^{2},
  \qquad r>0,
\end{equation}
and consequently
\begin{equation}\label{eq:crossing}
  \Psi_{R}(r)>\Psi_{R}(R)\quad\text{for }0<r<R,
  \qquad
  \Psi_{R}(r)<\Psi_{R}(R)\quad\text{for }r>R.
\end{equation}
\end{lemma}

\begin{proof}
From \eqref{eq:wdef},
\[
w(r)=\frac{1-\varphi'(r)}{n-1}.
\]
Thus, $\Psi_{R}(r)$ takes the form
\[
\Psi_{R}(r)=\varphi'(r)^{2}+q(r)^{2}+2\varphi'(R)(1-\varphi'(r)),
\]
and consequently
\[
\Psi_{R}(r)-\Psi_{R}(R)
= \varphi'(r)^{2}-2\varphi'(R)\varphi'(r)+\varphi'(R)^{2}+q(r)^{2}-q(R)^{2}.
\]
which is \eqref{eq:Psiidentity}. Since $\varphi'$ and $q$ strictly decrease, we obtain $\varphi'(r)>\varphi'(R)$ and $q(r)>q(R)$ for $r<R$. Therefore, the right-hand side of \eqref{eq:Psiidentity} is positive for $r<R$.

Let now $r>R$. Since $q'<0$, the pointwise bound $0<\frac{\varphi''}{q'}\leq 1$ implies $0<-\varphi''\leq -q'$. Integrating from $R$ to $r$ yields
\[
0< \varphi'(R)-\varphi'(r)\leq q(R)-q(r).
\]
Hence, using \eqref{eq:Psiidentity} and $q(r)>0$,
\[
\begin{aligned}
\Psi_{R}(R)-\Psi_{R}(r)
&= q(R)^{2}-q(r)^{2}-\left(\varphi'(R)-\varphi'(r)\right)^{2} \\
&\geq \left(q(R)-q(r)\right)\left(q(R)+q(r)\right)-\left(q(R)-q(r)\right)^{2} \\
&= 2\,q(r)\,\left(q(R)-q(r)\right) > 0,
\end{aligned}
\]
which is the second half of \eqref{eq:crossing}.
\end{proof}

\begin{remark}\label{rem:whynotmonotone}
Lemma~\ref{lem:crossing} is strictly weaker than the monotonicity of $\Psi_{R}$, and this is what makes it valid in every dimension. The argument in \cite{GuLiWan2025} requires $\mathcal{E}+\frac{2(n-1)}{n}w$ to be decreasing on the entire interval $(0,\infty)$. Based on the computation in \cite{GuLiWan2025}*{Lemma~4.4}, this is equivalent to the global bound $\lambda\lambda'\varphi'/\varphi\leq\frac{2n-1}{n-1}$. In $\mathbb{H}^{3}$, the left-hand side of this inequality is unbounded, which means that this approach is invalid for $n=3$. By contrast, the crossing condition \eqref{eq:crossing} holds for all $n\geq 3$ and all $R>0$.
\end{remark}

\section{Proof of the main theorems}\label{sec:proof}

\begin{proof}[Proof of Theorem~\ref{thm:starshaped}]
Set
\[
  \mathcal{W}:=\mathcal{W}(\Omega),\qquad
  \mathcal{W}_0:=\mathcal{W}(B(R)).
\]
By Proposition~\ref{prop:rayleigh} and Theorem~\ref{thm:moment},
\begin{equation}\label{eq:sigma_upper}
  |\Sigma|\,|\Omega|^{2}\eta(\mathcal{W})^{2}\sigma_1(\Omega)
  \leq \int_{\Omega}\mathcal{E}\dd v.
\end{equation}

Since $|\Omega|=|B(R)|$, Lemma~\ref{lem:crossing} gives $\Psi_{R}<\Psi_{R}(R)$ on $\Omega\setminus B(R)$ and $\Psi_{R}>\Psi_{R}(R)$ on $B(R)\setminus\Omega$. Hence
\begin{equation}\label{eq:shell_inequality}
 \int_{\Omega}\Psi_{R}-\int_{B(R)}\Psi_{R}
  =\int_{\Omega\setminus B(R)}\Psi_{R}-\int_{B(R)\setminus\Omega}\Psi_{R}
  \leq0,    
\end{equation}
with equality if and only if $|\Omega\,\triangle\,B(R)|=0$. Substituting the definition of $\Psi_R$ into \eqref{eq:shell_inequality} and using \eqref{eq:E0}, we obtain 
\begin{equation}\label{eq:step3}
  \int_{\Omega}\mathcal{E}\dd v + 2(n-1)\varphi'(R)\,\mathcal{W} \leq \varphi'(R)|B(R)| + 2(n-1)\varphi'(R)\mathcal{W}_0.
\end{equation}

Since $w$ is strictly increasing (Lemma~\ref{lem:phimono}) and $|\Omega|=|B(R)|$, Theorem~\ref{thm:transplant} gives $\mathcal{W}\geq\mathcal{W}_0$. By convexity of $\eta^{2}$ and
\eqref{eq:etatangentdata},
\begin{equation}\label{eq:step4a}
 \begin{split}
  \eta(\mathcal{W})^{2}
  &\geq\frac{1}{|S(R)|^{2}}
   -\frac{2(n-1)}{|B(R)|\,|S(R)|^{2}}(\mathcal{W}-\mathcal{W}_0) \\
  &=\frac{1}{|B(R)|\,|S(R)|^{2}}
   \left(|B(R)|-2(n-1)(\mathcal{W}-\mathcal{W}_0)\right).
 \end{split}
\end{equation}
Rearranging \eqref{eq:step3} gives
$\int_{\Omega}\mathcal{E}\dd v
 \leq \varphi'(R)\left(|B(R)|-2(n-1)(\mathcal{W}-\mathcal{W}_0)\right)$, so \eqref{eq:step4a} becomes
\begin{equation}\label{eq:step4b}
  \eta(\mathcal{W})^{2}
  \geq\frac{\int_{\Omega}\mathcal{E}\dd v}
         {\varphi'(R)\,|B(R)|\,|S(R)|^{2}} > 0.
\end{equation}
Combining \eqref{eq:sigma_upper} and \eqref{eq:step4b}, and applying \eqref{eq:ballsigma} with $\varphi(R)=|B(R)|/|S(R)|$, we obtain
\[
\begin{aligned}
  |\Sigma|\,\sigma_1(\Omega)
  &\leq\frac{\int_{\Omega}\mathcal{E}\dd v}
         {|\Omega|^{2}\eta(\mathcal{W})^{2}} \\
  &\leq\frac{\varphi'(R)|S(R)|^{2}}{|B(R)|}
   =|S(R)|\frac{\varphi'(R)}{\varphi(R)}
   =|S(R)|\,\sigma_1(B(R)),
\end{aligned}
\]
which establishes \eqref{eq:main-product}.

If equality holds in \eqref{eq:main-product}, then equality holds in \eqref{eq:step3}, which forces $|\Omega\,\triangle\,B(R)|=0$. Since $\Omega$ is open with smooth boundary, this means $\Omega=B(R)$.
\end{proof}

\begin{proof}[Proof of Theorem~\ref{thm:main}]
Let $\rho$ be defined by $|B(\rho)|=|\Omega|$. 
Theorem~\ref{thm:isop} gives $|\partial\Omega|\geq|S(\rho)|$, and since $t\mapsto|S(t)|$ is strictly increasing and $|S(R)|=|\partial\Omega|$ (by definition of $\Omega^*=B(R)$), we have $\rho\leq R$. By 
Theorem~\ref{thm:starshaped} and Proposition~\ref{prop:ball},
\[
|\partial\Omega|\,\sigma_1(\Omega)
\leq|S(\rho)|\,\sigma_1(B(\rho))
\leq|S(R)|\,\sigma_1(B(R))
=|\partial\Omega|\,\sigma_1(\Omega^*).
\]
Dividing by $|\partial\Omega|>0$ gives \eqref{eq:main}. If equality holds, then equality holds in both inequalities above. Theorem~\ref{thm:starshaped} forces $\Omega=B(\rho)$, and the boundary area identifies $\rho=R$.
\end{proof}

\section{A harmonic mean inequality}\label{sec:harmonic}

In this final section, we present the following harmonic mean inequality for the first $(n-1)$ non-zero Steklov eigenvalues, which extends \cite{GuLiWan2025}*{Corollary~1.6} to the case $n=3$.

\begin{proposition}\label{prop:harmonic}
Let $n\geq3$ and let $\Omega\subset\mathbb{H}^{n}$ be a smooth bounded domain. If $\partial\Omega$ is star-shaped and mean convex, then
\[
  \sum_{i=1}^{n-1}\frac{1}{\sigma_{i}(\Omega)}
  \geq\sum_{i=1}^{n-1}\frac{1}{\sigma_{i}(\Omega^*)},
\]
where $\Omega^*$ is the geodesic ball with $|\partial\Omega^*|=|\partial\Omega|$. Equality holds if and only if $\Omega$ is a geodesic ball.
\end{proposition}

\begin{proof}
In the proof of Theorem~\ref{thm:starshaped}, the balance condition is used only to set up the test functions via Proposition~\ref{prop:rayleigh}. The rest of the proof requires merely that $\partial\Omega$ is star-shaped 
and mean convex. From those parts one obtains
\[
  \frac{\int_{\Omega}\mathcal{E}\dd v}{\int_{\Sigma}\varphi^{2}\dd\mu}
  \leq\sigma_1(\Omega^*).
\]
Verma~\cite{Verma2021}*{eq.~(10)} showed that after a suitable re-centring,
\[
  \int_{\Sigma}\varphi^{2}\dd\mu
  \le\frac{1}{n-1}\sum_{i=1}^{n-1}\frac{1}{\sigma_{i}(\Omega)}
     \int_{\Omega}\mathcal{E}\dd v.
\]
Since $\sigma_1(\Omega^*)=\cdots=\sigma_{n}(\Omega^*)$, the two inequalities combine to give the claimed bound. The equality case follows from the equality case of Theorem~\ref{thm:starshaped}. 
\end{proof}

\section*{Acknowledgments}
The authors thank Professor Bobo Hua for his guidance and constant support. The second author is supported by NSFC (Nos.~12101125 and 12371052) and the Fujian Alliance of Mathematics (No.~2024SXLMMS01). The third author is partially supported by the National Key R\&D Program of China (No.~2025YFA1017500).

\section*{Declarations}

\noindent\textbf{Conflict of Interest}\ \ The authors declare that they have no conflict of interest.

\vspace{1em}
\noindent\textbf{Data Availability}\ \ Data sharing is not applicable to this article as no datasets were generated or analyzed during the current study.

\vspace{1em}
\noindent\textbf{AI assistance statement}\ \  The authors used OpenAI's ChatGPT with the GPT-5.6 Sol model to assist with initial conceptualization, symbolic checking and manuscript editing. All mathematical statements and proofs were independently verified by the authors, who take full responsibility for the content and accuracy of the article.

\raggedbottom
\bibliographystyle{plain}
\bibliography{weinstock}

\end{document}